\documentclass[fontsize=11pt,parskip=full]{scrartcl}
\usepackage{amsmath,amssymb,amsfonts,amsthm,mathabx}
\usepackage[dvips]{graphics}
\usepackage{enumerate}
\usepackage{mathrsfs}
\usepackage{bbm}
\usepackage{subfigure}
\usepackage{color}

\usepackage[T1]{fontenc}

\DeclareSymbolFont{AMSb}{U}{msb}{m}{n}

\def\B{\mathcal {B}}

\def\C{\mathcal {C}}

\def\P{\mathbb P}
\def\N{\mathbb N}

\def\L{\mathcal {L}}
\def\M{\mathcal {M}}

\def\R{\mathbb {R}}

\def\dd{\mathrm {d}}

\def\EE{\mathcal {E}}

\def\FF{\mathscr{F}}

\def\E{\mathbb {E}}

\def\1{\mathbbm{1}}

\def\scirc{\mathbin{\raise.15ex\hbox{\scriptsize$\circ$}}}

\newtheorem{thm}{Theorem}[section]
\newtheorem{prop}[thm]{Proposition}
\newtheorem{cor}[thm]{Corollary}

\theoremstyle{definition}

\newtheorem{remark}[thm]{Remark}

\newtheorem{ex}[thm]{Example}

\title{Couplings, generalized couplings and uniqueness of invariant measures}

\author{Michael Scheutzow%
  \thanks{Institut f\"ur Mathematik, MA 7-5, Fakult\"at II, 
    Technische Universit\"at Berlin, 
    Stra\ss e des 17.~Juni 136, 10623 Berlin, FRG;  \ 
    \small\texttt{ms{\scriptsize @}math.tu-berlin.de}}
}

\date{}

\begin{document}\maketitle

\begin{abstract}\noindent
  We provide sufficient conditions for uniqueness of an invariant probability measure of a Markov kernel in terms of (generalized) couplings.
  Our main theorem generalizes previous results which require the state space to be Polish. We provide an example showing that uniqueness can fail
  if the state space is separable and metric (but not Polish) even though a  coupling defined via a continuous and positive definite function exists.

\par\medskip

\end{abstract}


\section{Introduction}
One important question in the theory of Markov processes is that of existence and uniqueness of invariant probability measures (ipms).
In this note we will concentrate on uniqueness. A sufficient condition for uniqueness of an ipm is provided by {\em Doob's theorem} based on 
appropriate equivalence assumptions of the transition probabilities. In fact, such kind of conditions even imply total variation convergence of all or almost all transition probabilities (see \cite{KS15}, \cite{SS20}). On the other hand, there are a number of cases in which an ipm is known to be unique and
for which it is also known that equivalence of transition probabilities fails, for example certain classes of stochastic functional differential equations, see e.g.~\cite{HMS11}.
In \cite[Theorem 1.1, Corollary 2.2]{HMS11} and later in \cite[Theorem 1, Corollary 1]{KS18}, the authors provided uniqueness
criteria in terms of {\em generalized couplings}. A basic assumption in both papers is that the state space is  {\em Polish} (i.e.~a separable
and completely metrizable topological space), a fact which is used in order to apply an ergodic decomposition theorem but also to guarantee inner regularity
of finite Borel measures. In recent years, there seems to be growing interest in invariant measures for Markov processes with non-Polish state space, like spaces of bounded measurable
functions (e.g.~\cite{BF20}).

In this note, we generalize previous results to (not necessarily separable) metric spaces and, in the Polish state space case, we allow that the distance function which appears in the coupling assumption, is a lower semi-continuous positive definite function and not necessarily a metric. We also provide an example
showing that this generalization fails to hold if the state space is separable and metric but not complete.

Let us briefly recall the previous approaches to show uniqueness via generalized couplings in the case of a Polish state space. Assume that a Markov kernel $P$ admits more than
one ipm. Then it is known that $P$ admits two distinct ergodic and hence mutually singular ipms $\mu$ and $\nu$ (see, e.g., \cite{H08}). Therefore, there  exist disjoint compact sets
$A$ and $B$ of $\mu$(resp.~$\nu$)-measure almost 1. By ergodicity, starting in $A$, the Markov chain will almost surely spend a large proportion of time in $A$ and similarly for $B$.
No matter how we couple the chains starting in $A$ and in $B$: most of the time, the first process is in $A$ and the second one is in $B$ and so their distance is at least
equal to the distance of the sets $A$ and $B$ (which is strictly positive), thus contradicting the usual {\em coupling} assumption that there exists a coupling for which the
processes starting in $A$ and in $B$ are very close for large times. This argument still holds if couplings are replaced by generalized couplings (see the definition below).  

The note is organized as follows. In the following section, we provide three elementary propositions, where the first and the third one constitute an elementary
substitute for the ergodic decomposition property which does not seem to be known for a  general state space. Then we present and prove the main result along the lines
\cite{HMS11} and  \cite{KS18} but using these propositions instead of ergodicity and inner regularity in the Polish case.  

\section{Preliminaries}

Let $P$ be a Markov kernel on the measurable space $(E,\EE)$. We denote the set of probability measures on $(E,\EE)$ by $\M_1(E,\EE)$ or just
$\M_1(E)$. If $\mu \in \M_1(E)$, then we write $\mu P$
for the image of $\mu$ under $P$. We are interested in providing criteria for the uniqueness of an invariant probability measure (ipm), i.e.~a probability measure
$\pi$ on $(E,\EE)$ satisfying $\pi P=\pi$.
We call two probability measures $\mu$ and $\nu$ on the measurable space $(E,\EE)$ {\em (mutually) singular}, denoted $\mu \perp \nu$, if
there exists a set $C \in \EE$ such that $\mu(C)=1$ and $\nu(C)=0$. As usual, $\mu \ll \nu$ means that the measure $\mu$ is absolutely
continuous with respect to $\nu$. If $E$ is a topological space, then we denote its Borel $\sigma$-field by $\B(E)$.

For $x \in E$, we denote the law of the chain with kernel $P$ and initial condition $x$ by $\P_x$. Note that $\P_x$ is a probability measure on the space
$\big(E^{\N_0},\EE^{\N_0}\big)$. $\C(\P_x,\P_y):=\big\{ \xi \in \M_1(E^{\N_0}\times  E^{\N_0}): \, \pi_1(\xi)=\P_x,\,\pi_2(\xi)=\P_y \big\}$ is called the set
of {\em couplings} of $\P_x$ and $\P_y$. Here, $\pi_i(\xi)$ denotes the image of  $\xi$ under the projection on the $i$-th coordinate, $i=1,2$. 
The set of {\em generalized couplings} $\hat \C (\P_x,\P_y)$ is defined as
$$
\hat \C(\P_x,\P_y):=\big\{ \xi \in \M_1(E^{\N_0}\times  E^{\N_0}): \, \pi_1(\xi)\ll\P_x,\,\pi_2(\xi)\ll\P_y \big\}. 
$$

The following elementary proposition is a consequence of the ergodic decomposition theorem   under the assumption that the space $(E,\EE)$ is standard Borel, i.e.~measurable isomorphic to a Polish space
equipped with its Borel $\sigma$-field, but we are not aware of a proof in the general case. 
\begin{prop}\label{singu}
Assume that $P$ admits more than one ipm. Then there exist two mutually singular ipm's.
\end{prop}

\begin{proof}
  Let $\mu$ and $\nu$ be two distinct ipm's. Assume first that $\mu$ and $\nu$ are mutually equivalent and define $f(x)=\frac{\dd \mu}{\dd \nu}(x)$, $x \in E$ and
  $A:=\{x\in E: \,f(x)>1\}$. Then $\mu(A),\nu(A) \in (0,1)$. 
We have (by invariance of $\nu$ and $\mu$)
  $$
\int_A P(y,A^c)\,\dd \nu(y)=\int_{A^c} P(y,A)\,\dd \nu(y)
$$
and
  $$
\int_A P(y,A^c)f(y)\,\dd \nu(y)=\int_{A^c} P(y,A)f(y)\,\dd \nu(y).
$$
Since $f(y)>1$ on $A$ and $f(y) \le 1$ on $A^c$, it follows that all four expressions in the two equations are in fact equal.
This implies $P(y, A^c)=0$ for ($\mu$ or $\nu$)-almost all $y \in A$ and hence $P(y,A)=0$ for almost all $y\in A^c$. 
Therefore, the probability measures $\frac 1{\mu(A)} \mu|_A$ and $\frac 1{\mu(A^c)} \mu|_{A^c}$ are mutually singular ipm's.\\

Let us now assume that $\mu$ and $\nu$ are distinct ipm's which are neither equivalent nor singular. Without loss of generality we assume that $\mu$ is not absolutely continuous with respect to $\nu$. Then there exist disjoint sets $A,B,C \in \EE$ such that $A \cup B \cup C=E$ and $\mu$ and $\nu$ restricted to $B$ are equivalent, $\nu(A)=0$ and $\mu(C)=0$.
By assumption, $\mu(A)>0$ and $\mu(B)$, $\nu(B)>0$. Then $P(x,B)=1$ for ($\mu$ or $\nu$-)almost all $x \in B$ showing that the normalized measures $\mu$ restricted to $B$ and to $A$
are mutually singular invariant probability measures.
\end{proof}

If $E$ is a non-empty set, $A$ and $B$ are subsets of $E$ and $\rho:E \times E \to [0,\infty)$, then we define
$$
\rho(A,B):=\inf\{\rho(a,b):\,a \in A,\,b\in B\},
$$
where the infimum over the empty set is defined as $+\infty$.
If $A=\{x\}$, then we write $\rho(x,B)$ instead of $\rho(\{x\},B)$. Further,
we call such a function
$\rho$ {\em positive definite}, if $\rho(x,y)=0$ iff $x = y$.

\begin{prop}\label{closed}
  Let $\mu$ and $\nu$ be probability measures on the Borel sets of a metric space $(E,d)$ such that $\mu \perp \nu$. Let $C \in \B(E)$ be such that $\mu(C)=1$ and $\nu(C)=0$.
  Then, for every $\varepsilon >0$, there exist closed sets $A\subset C$ and $B\subset C^c$ such that $\mu(A)>1-\varepsilon$, $\nu(B)>1-\varepsilon$  and $d(A,B)>0$.

  If $E$ is Polish and $d$ is a (not necessarily complete) metric which generates the topology of $E$, then, in addition, $A$ and $B$ can be chosen to be compact.
  In this case, it holds that for any lower semi-continuous and positive definite function $\rho:E \times E \to [0,\infty)$, we have $\rho(A,B)>0$.
\end{prop}

\begin{proof}
  By \cite[Lemma 7.2.4.]{C13},
  there exists a closed set $A\subseteq C$ such that $\mu(A)> 1-\varepsilon$. Similarly, there is a closed set $B_0 \subset C^c$ for which
  $\nu(B_0)> 1-\varepsilon/2$. For $n \in \N$, let $B_n:=\{y \in E:\,d(y,A)\ge 1/n\}$. Choose $n\in \N$ such that $\nu(B_n\cap B_0)> 1-\varepsilon$. Then
  $A$ and $B:=B_n\cap B_0$ satisfy all properties stated in the proposition (and $d(A,B)\ge 1/n$).

  On a Polish space, every finite measure on the Borel sets is  {\em regular} (\cite[Proposition 8.1.12]{C13}) and therefore, there exist compact sets
  $A\subset C$ and $B \subset C^c$ such that $\mu(A)> 1-\varepsilon$ and $\mu(B)> 1-\varepsilon$. Since $A$ and $B$ are disjoint, we have
  $d(A,B)>0$. Moreover, if $\rho:E \times E \to [0,\infty)$ is lower semi-continuous and positive definite, then, automatically, $\rho(A,B)>0$
  by compactness of $A$ and $B$.
\end{proof}

\begin{prop}\label{psi}
  Let $\mu$ be an invariant probability measure of the Markov kernel  $P$ on the measurable space $(E,\EE)$ and let $f: E \to \R$ be bounded and measurable. For $\gamma \in \R$
  define $\psi_\gamma:E \to [0,1]$ by
  \begin{equation}\label{psigamma}
  \psi_\gamma(x)=\P_x\Big( \liminf_{n \to \infty} \frac 1n \sum_{i=0}^{n-1}f(X_i) \ge \gamma\Big).
  \end{equation}
  Then, $\psi_\gamma(x)\in \{0,1\}$ for $\mu$-almost all $x \in E$.

If, moreover,  $f(x)\in [0,1]$ for all $x \in E$, $m:=\int f\,\dd \mu$, and $\gamma \in [0,m]$, then 
  $$
\mu\big(\big\{x:\,\psi_\gamma(x)=1\big\}\big)\ge 1-\frac {1-m}{1-\gamma}.
  $$
\end{prop}

\begin{proof}
  Let $X_0,\,X_1,...$ be the Markov chain started with $\L(X_0)=\mu$ defined on a space $(\Omega,\FF,\P)$. Then $\psi_\gamma(X_n),\,n \in \N_0$ is a stationary process and a (bounded) martingale with respect to
  the complete filtration $(\FF_n)$ generated by $(X_n)$, so
  $Z:=\lim_{n \to \infty}\psi_\gamma(X_n)$ exists almost surely by the martingale convergence theorem. Stationarity  implies that $n \mapsto \psi_\gamma(X_n)$ is almost surely constant. Further, $Z$ is $\FF_\infty$-measurable
  and therefore $Z \in \{0,1\}$ almost surely. Hence, $\psi_\gamma(x) \in \{0,1\}$ for $\mu$-almost all $x \in E$.\\

  To establish the final statement, we apply Birkhoff's ergodic theorem to see that
  $$
Y:=\lim_{n \to \infty}\frac 1n \sum_{i=0}^{n-1}f(X_i)
$$
exists almost surely and $\E Y=m$. Therefore, by Markov's inequality,
$$
\mu\big(\big\{x:\,\psi_\gamma(x)=1\big\}\big)=\P (Y\ge \gamma)=1-\P (1-Y >1-\gamma)\ge 1- \frac {1-m}{1-\gamma},
$$
so the proof is complete.
\end{proof}

\begin{remark}
  Note that, due to Birkhoff's ergodic theorem, we could replace the $\liminf$ in \eqref{psigamma} by $\limsup$ or $\lim$. This changes the value of $\psi_\gamma$ only on a set
  of $\mu$-measure 0. 
  \end{remark}

  \section{Main result}

  Before we state the main result we address a small technical issue. If the metric space $(E,d)$ is not separable, then it may happen that the map $(x,y)\mapsto d(x,y)$ is not
  $\B(E)\otimes\B(E)$-measurable (the map is of course $B(E\times E)$-measurable but  $\B(E)\otimes\B(E)$ may be strictly contained in $\B(E\times E)$). If $\xi$ is a probability measure on
  $(E \times E,\EE \otimes \EE)$, then we silently assume that an expression like $\xi(A)$ is interpreted as $\xi^*(A)$ in case $A$ is not measurable where $\xi^*$ denotes the
  outer measure associated to $\xi$.
  
  \begin{thm}\label{main}
    Let $\mu_1$ and $\mu_2$ be invariant probability measures of the Markov kernel $P$ on the metric space $(E,d)$ with Borel $\sigma$-field
    $\EE:=\B(E)$. Assume that there exists a set $M \in \EE \otimes \EE$ such that $\mu_1 \otimes \mu_2(M)>0$ and that for every $(x,y) \in M$
    there exists some $\alpha_{x,y}>0 $ such that for
    every $\varepsilon>0$ there exists some $\xi_{x,y}^\varepsilon \in \hat \C(\P_x,\P_y)$ such that
    \begin{equation}\label{formula}
 \xi_{x,y}^\varepsilon\Big( (\xi,\eta)\in E^{\N_0}\times E^{\N_0}:\,\limsup_{n \to \infty} \frac 1n\sum_{i=0}^{n-1} \1_{[0,\varepsilon]}\big(d(\xi_i,\eta_i)\ge \alpha_{x,y}\big)>0.
    \end{equation}  
Then $\mu_1$ and $\mu_2$ cannot be mutually singular.

If, moreover, $E$ is Polish 
and $\rho:E \times E \to [0,\infty)$ is a lower semicontinuous and positive definite function for which
\eqref{formula} holds for $d$ replaced by $\rho$ then, again, $\mu_1$ and $\mu_2$ cannot be mutually singular.
    \end{thm}

The following corollary is a simple consequence of Theorem \ref{main} and Proposition \ref{singu}. 
    
    \begin{cor}
Let  $P$ be a Markov kernel on the metric space $(E,d)$ with Borel $\sigma$-field
$\EE:=\B(E)$.      Assume that there exists a set $M \in \EE$ such that $\mu(M)>0$ for every invariant probablity measure $\mu$  and that for
every $x,y \in M$ there exists $\alpha_{x,y}>0$ such that for every $\varepsilon>0$ there exists some $\xi_{x,y}^\varepsilon \in \hat \C(\P_x,\P_y)$ such that
\begin{equation}\label{formula2}
 \xi_{x,y}^\varepsilon\Big( (\xi,\eta)\in E^{\N_0}\times E^{\N_0}:\,\limsup_{n \to \infty} \frac 1n\sum_{i=0}^{n-1}\1_{[0,\varepsilon]}\big(\rho(\xi_i,\eta_i)\ge \alpha_{x,y}\big)>0,
\end{equation}
where either $\rho=d$, or $\rho$ is lower semicontinuous and positive definite and $E$ is Polish,
then there exists at most one invariant probability measure.
   \end{cor}

   \begin{remark}
     Conditions  \eqref{formula} and  \eqref{formula2} are slightly weaker than \cite[(2.5)]{KS18}: condition \eqref{formula} is of the form
     $\P\big(\limsup_{n \to \infty}Z_n\ge \alpha\big)>0$ while (2.5) in \cite{KS18} is of the form $\limsup_{n \to \infty}\E Z_n\ge \alpha$.
     \end{remark}
   
      \begin{proof}[Proof of Theorem \ref{main}]
        Assume that $\mu_1$ and $\mu_2$ are mutually singular invariant probability measures of $P$. Let $C \in \EE$ be a set such that $\mu_1(C)=1$
        and $\mu_2(C)=0$. By Proposition \ref{closed}, there exist closed sets
        $A \subset C$ and $B \subset C^c$ such that $\mu_1(A)>1-\kappa$, $\mu_2(B)>1-\kappa$, and $\rho(A,B)>0$ with $\rho:=d$ if $E$ is not Polish.

Denoting the chain starting at $X_0=x\in E$ by $(X_i^x), i \in \N_0$, we have,  by Proposition \ref{psi}, 
\begin{align}\label{absch}
\begin{split}
  \mu_1\Big(\big\{x \in E:\,\liminf_{n \to \infty} \frac 1n \sum_{i=1}^{n-1} \1_A(X_i^x) &\ge \gamma, \,\P_x\mbox{-a.s.}\big\} \Big) 
                                                                                        >1-\frac \kappa{1-\gamma},\\
  \mu_2\Big(\big\{x \in E:\,\limsup_{n \to \infty} \frac 1n \sum_{i=1}^{n-1} \1_B(X_i^x) &\ge \gamma,\,\P_x\mbox{-a.s.}\big\} \Big) 
                                                                                        >1-\frac \kappa{1- \gamma},
\end{split}                                                                                      
\end{align}
where $\gamma \in (0,1)$.

We now proceed to assign specific values to the variables $\gamma$ and $\kappa$.

Note that there exist some $\delta,\bar \delta >0$ such that for every set $\bar M \in \EE\otimes \EE$, $\bar M \subset M$ such that $\mu_1\otimes \mu_2(\bar M)\geq \mu_1\otimes \mu_2(M)-\delta$
there exists some $(x,y)\in \bar  M$ such that $\alpha_{x,y}>\bar \delta$ (even if $(x,y)\mapsto \alpha_{x,y}$ is non-measurable). Fix such $\delta,\bar \delta >0$ and fix
$\gamma \in (0,1)$ such that
$$
2(1-\gamma)<\bar \delta.
$$
Define
\begin{align*}
  E_1:=\Big\{x \in E: &\,\liminf_{n \to \infty} \frac 1n \sum_{i=1}^{n-1} \1_A(X_i^x) \ge \gamma,\, \P_x\mbox{-a.s.}\Big\},\\
  E_2:=\Big\{y \in E: &\,\limsup_{n \to \infty} \frac 1n \sum_{i=1}^{n-1} \1_B(X_i^y)\ge  \gamma,\, \P_y\mbox{-a.s.}\Big\}.
\end{align*}
The sets $E_1$ and $E_2$ still depend on $\kappa$ via $A$ and $B$. Using \eqref{absch}, we can find (and fix) $\kappa>0$  such that $\mu_1\otimes \mu_2\big(E_1\times E_2\big) =\mu_1\big(E_1\big)\mu_2\big(E_2\big)\ge 1-\delta$, so $\bar M:=\big(E_1 \times E_2\big)\cap M$ satisfies $\mu_1\otimes \mu_2(\bar M)\geq \mu_1\otimes \mu_2(M)-\delta$.
Therefore, there exists $(x,y)\in \big(E_1 \times E_2\big)\cap M$ such that $\alpha_{x,y}> \bar \delta$. Fix such a pair $(x,y)$ and let $\varepsilon:=\frac 12 \rho(A,B)$.
Pick $\xi_{x,y}^\varepsilon\in \hat \C\big(\P_x,\P_y\big)$ as in the theorem. If $\big(\hat X_i,\hat Y_i\big)_{i \in \N_0}$ has law $\xi_{x,y}^\varepsilon$, then  $\L (\hat X)\ll \P_x$ and
$\L (\hat Y)\ll \P_y$ and so
$$
\liminf_{n \to \infty}\frac 1n \sum_{i=1}^{n-1} \1_A(\hat X_i) \ge \gamma,\, \liminf_{n \to \infty}\frac 1n \sum_{i=1}^{n-1} \1_B(\hat Y_i) \ge \gamma, \mbox{ a.s.}.
$$
Therefore,
\begin{equation}\label{fastfertig}
\liminf_{n \to \infty}\frac 1n \sum_{i=1}^{n-1} \1_{A \times B}(\hat X_i,\hat Y_i) \ge 2\gamma -1>1-\bar \delta > 1-\alpha_{x,y}\mbox{ a.s.}. 
\end{equation}
Since
$$
\1_{[0,\varepsilon]}\big(\rho\big(\hat X_i,\hat Y_i\big)\big)\le 1- \1_{A \times B}\big(\hat X_i,\hat Y_i\big),
$$
we see that \eqref{fastfertig} contradicts assumption \eqref{formula},  so there cannot exist two mutually singular invariant probability measures.
\end{proof}

     \section{A counterexample}
     The basic set-up of the following example is inspired by \cite[Example 1]{BPR15} in which the authors show that the ``gluing lemma'' need not hold on a separable and
     metrizable space. Our example shows that even if there exists a continuous and positive definite function $\rho:E\times E$, where $E$ is separable and metric,
     such that for every pair $x,y\in E$ there exists a (true) coupling $(X_n,Y_n)$
     for which $\rho(X_n,Y_n)$ converges to 0 almost surely, uniqueness of an invariant probability measure may not hold. 
    \begin{ex}
      Let $I \subset [0,1]$ be a set such that $\lambda^*(I)=1$ and $\lambda_*(I)=0$, where $\lambda$ denotes Lebesgue measure on the Borel sets of $[0,1]$ and $\lambda^*$ and
      $\lambda_*$ are the corresponding outer and inner measures. Further, 
      let $J:=[0,1]\backslash I$ (then  $\lambda^*(J)=1$ and $\lambda_*(J)=0$). Let $E$ be the disjoint union of $I$ and $J$,
      i.e.~$E=E_1\cup E_2$, where $E_1:=\{(x,1):\,x \in I\}$ and $E_2:=\{(x,2):\,x \in J\}$ equipped with the metric
      $$
      d(x,y)=\left\{
        \begin{array}{ll}
          |x-y|& \mbox{ if } (x,y)\in E_1\times E_1 \mbox{ or } (x,y) \in E_2\times E_2,\\
          1&\mbox{ if } (x,y)\in E_1\times E_2 \mbox{ or } (x,y) \in E_2\times E_1.
        \end{array}
        \right. 
        $$
        Note that $E$ is separable (but not Polish since otherwise the following construction could not work). 
        We define $\rho:E \times E\to [0,1]$ as $\rho\big((x,i),(y,j)\big)=|x-y|$ for  $(x,i) \in E_i$, $(y,j) \in E_j$, $i,j \in \{1,2\}$.
        Obviously, $\rho$ is continuous. Further, $\rho$ is positive definite since $\rho\big((x,i),(y,j)\big)=0$ implies that $i=j$ and hence
        either both $x$ and $y$ are in $I$ or both $x$ and $y$ are in $J$ (since $I$ and $J$ are disjoint). In fact, $\rho$ is a (continuous) metric on $E$ which makes
        $(E,\rho)$ a Polish space (which is isometric to the interval $[0,1]$ equipped with the Euclidean metric). Note that the topology generated by $\rho$ is different
        from the one generated by $d$. 

        Next, we construct an $E$-valued Markov chain with two different invariant measures $\mu$ and $\nu$ and a coupling $(X_n,Y_n)$ of two
        copies of the chain starting at $(x,y)$ such that $\lim_{n \to \infty}\rho(X_n, Y_n)=0$ almost surely.

        For $A \subset E$, we define $\pi_1(A):=\{x \in I:\,(x,1)\in A\}$ and $\pi_1(A):=\{x \in J:\,(x,2)\in A\}$. Let
        $$
        \mu(A):=\lambda^*\big(\pi_1(A)\big),\;\nu(A):=\lambda^*\big(\pi_2(A)\big),\;A \in \B (E).
        $$
        We define the Markov kernel $P$ on $E$ by
        $$
        P(x,.)=\left\{
          \begin{array}{ll}
            \mu,& \mbox{if } x \in E_1\\
            \nu,& \mbox{if } x \in E_2.
          \end{array}
        \right.
        $$
        Clearly, $\mu$ and $\nu$ are mutually singular invariant probability measures  of $P$. Note that conditional on $X_0=x\in E_1$ (resp.~$E_2$)
        the sequence $X_1,X_2,...$ is i.i.d.~with law $\mu$ (resp.~$\nu$).

        We define $\xi_{x,y}\in \C\big(\P_x,\P_y\big)$ as follows. If $x,y$ are both in $E_1$, then we let $X_1,X_2,...$ be i.i.d.~with law $\mu$ and $Y_i:=X_i$, $i \in \N$ and similarly if
        $x,y$ are both in $E_2$. This defines a coupling $\xi_{x,y}\in \C\big(\P_x,\P_y\big)$ which satisfies $\lim_{n \to \infty}\rho(X_n,Y_n)=0$.

        Now we assume that $x\in E_1$ and $y \in E_2$. We let $(X_1,Y_1),\,(X_2,Y_2),...$ be independent with a distribution depending on $n \in \N$ as follows.
        For given $n \in \N$, we consider a random variable $U$ which is uniformly distributed on $\{0,...n-1\}$. Let $X_n$ and $Y_n$ be conditionally independent given $U$ with law
$$
\P\big( X_n \in A,Y_n\in B|U=i\big)=n^2\lambda^*\Big(\pi_1(A)\cap\Big[ \frac in,\frac {i+1}n\Big)\Big)\cdot \lambda^*\Big(\pi_2(A)\cap\Big[ \frac in,\frac {i+1}n\Big)\Big),\quad A,B \in \B(E).
$$
Clearly, this defines a coupling of $\P_x$ and $\P_y$ for which $\rho(X_n,Y_n)\le \frac 1n$ almost surely. 
      \end{ex}

      \begin{remark} Note that the Markov kernel $P$ in the previous example is even {\em strong Feller}, i.e.~the map $x \mapsto \int f(y)\,P(x,\dd y)$ is continuous for every
        bounded measurable function $f:E \to \R$. 
        \end{remark}

\end{document}